\documentclass[a4paper,12pt]{amsart}

\pdfoutput=1
\usepackage[hidelinks,pdfstartview=FitH]{hyperref} 
\usepackage{amssymb,tikz}
\usepackage[all,cmtip]{xy}
\usepackage{verbatim}
\usepackage{fullpage}
\usepackage{todonotes}
\usepackage{soul}

\allowdisplaybreaks[4]

\newtheorem{prop}{Proposition}[section]
\newtheorem{lemma}[prop]{Lemma}
\newtheorem{rem}[prop]{Remark}
\newtheorem{thm}[prop]{Theorem}
\newtheorem{cor}[prop]{Corollary}
\newtheorem{df}[prop]{Definition}

\newtheorem*{thm*}{Theorem}
\newtheorem*{corollary*}{Corollary}

\numberwithin{equation}{section}

\newcommand{\Z}{\mathbb{Z}}

\newcommand{\C}{\mathbb{C}}

\newcommand{\bS}{S}
\newcommand{\bB}{B}
\renewcommand{\emptyset}{\varnothing}
\newcommand{\id}{\mathrm{id}}

\title[Trimmable graph C*-algebras]{\vspace*{-15mm}An equivariant pullback structure of\\ 
trimmable graph C*-algebras} 

\author[F.~Arici]{Francesca Arici} 
\address[F.~Arici]{Max-Planck-Institut f\"ur Mathematik in den Naturwissenschaften, Inselstr.\ 22, 04103 Leipzig, Leipzig, Germany.}
\email{francesca.arici@mis.mpg.de}
 
\author[F.~D'Andrea]{Francesco D'Andrea} 
\address[F.~D'Andrea]{Universit\`a di Napoli ``Federico II'' and I.N.F.N. Sezione di Napoli, Complesso MSA, Via Cintia, 80126 Napoli, Italy.}
\email{francesco.dandrea@unina.it}

\author[P. M.~Hajac]{Piotr M.~Hajac}
\address[P. M.~Hajac]{Instytut Matematyczny, Polska Akademia Nauk, ul. \'Sniadeckich 8, Warszawa, 00-656 Poland}
\email{pmh@impan.pl }

\author[M.~Tobolski]{Mariusz Tobolski}
\address[M.~Tobolski]{Instytut Matematyczny, Polska Akademia Nauk, ul. \'Sniadeckich 8, Warszawa, 00-656 Poland}
\email{mtobolski@impan.pl }


\begin{document}
\baselineskip14.75pt
\parskip=0.5\baselineskip
\parindent=7.5mm

\begin{abstract}
We prove that the graph C*-algebra $C^*(E)$ of a trimmable
 graph $E$ is $U(1)$-equi\-va\-riant\-ly isomorphic to a pullback C*-algebra of
a subgraph C*-algebra $C^*(E'')$ and the C*-algebra of functions on a circle tensored with another 
subgraph C*-algebra~$C^*(E')$. This allows us to unravel the structure and K-theory of the fixed-point subalgebra 
$C^*(E)^{U(1)}$ through the (typically  simpler) C*-algebras~$C^*(E')$, $C^*(E'')$ and $C^*(E'')^{U(1)}$.
As examples of trimmable graphs, we consider one-loop extensions of 
the standard graphs encoding respectively the
Cuntz algebra $\mathcal{O}_2$ and the Toeplitz algebra $\mathcal{T}$. Then we analyze equivariant 
pullback structures of 
trimmable graphs yielding  the C*-algebras of the 
Vaksman--Soibelman quantum sphere $S^{2n+1}_q$ and  the quantum 
lens space~$L_q^3(l; 1,l)$, respectively. 
\end{abstract}
\maketitle
\vspace*{-5mm}
\tableofcontents
%\clearpage

\section{Introduction and preliminaries}%\vspace*{15mm}
\noindent
Graph C*-algebras are remarkable examples of 
``operator algebras that one can see"~\cite{R05}.
In particular, they proved to be extremely useful in determining the K-theory of noncommutative deformations
of interesting topological spaces. They come equipped with a natural circle action (called the gauge action),
and their gauge-invariant subalgebras describe some fundamental examples of noncommutative 
topology~\cite{BS16,HS02,HS03}.

The goal of this paper is to present $U(1)$-C*-algebras from a large class of trimmable graph C*-algebras
as $U(1)$-equivariant pullbacks, so as to determine a pullback structure of their fixed-point subalgebras.
The pullback structure yields a  Mayer--Vietoris six-term exact sequence in K-theory allowing us 
to express the even-K-group of the fixed-point subalgebra of a trimmable graph
C*-algebra in terms of the  K-theory  of simpler C*-algebras.
More precisely, we have:
\begin{thm*}
Let $E$ be a $\bar{v}$-trimmable graph. Denote by $E'$ the subgraph of $E$ obtained  
by removing the  unique outgoing edge (loop) of~$\bar{v}$, and by $E''$ the subgraph of $E'$
obtained  by removing the vertex $\bar{v}$ and all edges ending in~$\bar{v}$.
Then there exist $U(1)$-equivariant *-homomorphisms making the following diagram
\begin{equation*}
\xymatrix{
&C^*(E)\ar[ld]\ar[rd]&\\
C^*(E'')\ar[dr]&& 
C^*(E')\otimes C(S^1)
\ar[dl]\\
&C^*(E'')\otimes C(S^1)&
}
\end{equation*}
a pullback diagram of $U(1)$-C*-algebras. Here $C^*(E)$ and $C^*(E'')$ are considered as
$U(1)$-C*-algebras with respect to the gauge action, whereas the tensor product algebras
are viewed as $U(1)$-C*-algebras with respect to the standard $U(1)$-action on~$C(S^1)$.
\end{thm*}
\begin{corollary*} 
The following Mayer--Vietoris six-term sequence in \mbox{K-theory} is exact.
\begin{equation*}
\xymatrixcolsep{2.5pc}
\xymatrix{
K_0(C^*(E)^{U(1)}) \ar[r]^-{}
& K_0(C^*(E'')^{U(1)})\oplus K_0(C^*(E')) \ar[r]^-{}
& K_0(C^*(E'')) \ar[d]\\
K_1(C^*(E''))\ar[u]^{\partial_{10}}
& K_1(C^*(E')) \ar[l]^-{}
& 0. \ar[l]^-{}}
\end{equation*}
\end{corollary*}

The K-theory of graph C*-algebras is very well developed by now \cite{Cu81,dt02,RTW04,R05,CET12}, 
so it is hard to make a substantial contribution in this area.
However, our approach to the matter does bring its own  benefits:
\vspace*{-2mm}
\begin{itemize}
\item
The K-theory of fixed-point subalgebras of higher-rank graph C*-algebras \cite{hnpsz18} 
can be quite difficult to study. 
It was shown in \cite{FHMZ17} how equivariant pullback structures of graph C*-algebras provided
by the above theorem can be used
to reduce the K-theory  computations for higher-rank graphs to far more doable K-theory  
calculations for usual graphs.\\
\vspace*{-3mm}\item
Although the K-groups of the gauge-invariant subalgebra of a graph C*-algebra are, in principle, computable
for any graph~\cite{Dr00}, the above corollary makes this computation more effective by shifting it to simpler
graphs.\\
\vspace*{-3mm}\item
When the gauge-invariant subalgebra of a graph C*-algebra is only Morita equivalent to a graph C*-algebra,
it is tricky to describe it explicitly, so having it as a pullback C*-algebra (see Corollary~\ref{main_fixed})
remedies the problem whenever the gauge-invariant subalgebra of the trimmed graph C*-algebra
is known.\\
\vspace*{-3mm}\item
The Milnor connecting homomorphism $\partial_{10}$ in the above corollary unravels
generators of the even K-group the gauge-invariant subalgebra of a graph C*-algebra, so it might 
useful even when the gauge-invariant algebra is again a graph C*-algebra.
\end{itemize}

A one-loop extension of any finite graph $E''$ 
automatically becomes a trimmable graph~$E$, and $E'$ is simply a one-sink extension 
of~$E''$~\cite{RTW04}. 
Thus we can ``grow'' trimmable  graphs from any graph,  and use
the above corollary to determine the even K-groups. In one case we take as $E''$ the standard graph encoding 
the Cuntz algebra~$\mathcal{O}_2$, and in another case we take as $E''$ the standard graph encoding 
the Toeplitz algebra~$\mathcal{T}$. Then, in both cases, the gauge-invariant subalgebra $C^*(E)^{U(1)}$ 
has infinite-rank $K_0$-group.

On the other hand, we use our approach to study trimmable graphs naturally occurring in noncommutative topology,
such as graphs giving  the C*-algebras of  the 
Vaksman--Soibelman quantum sphere $S^{2n+1}_q$ and  the quantum 
lens space~$L_q^3(l; 1,l)$, respectively. The gauge-invariant subalgebra of the former C*-algebra defines
the Vaksman--Soibelman quantum complex projective space
 $\mathbb{C}{\rm P}^n_q$, and the gauge-invariant subalgebra of the latter C*-algebra defines
the quantum teardrop  $\mathbb{W}{\rm P}^1_q(1,l)$.
Although much is already known (generators included) about both the K-theory of
$C(\mathbb{C}{\rm P}^n_q)$  and $C(\mathbb{W}{\rm P}^1_q(1,l))$, revealing pullback structures
of these C*-algebras allows us to view projections whose classes generate the even K-groups
as Milnor idempotents, thus rendering the noncommutative vector bundles they define ``locally trivial''
and given by a clutching of trivial noncommutative vector bundles.

\subsection{Introduction}

Pushout diagrams in topology provide a systematic way of ``gluing'' of topological spaces,
in particular in the category of compact Hausdorff spaces. The latter is dualized by the 
Gelfand transform to the category of commutative unital C*-algebras with pushout diagrams
of spaces translated into pullback diagrams of algebras.

Our motivating example came from the following $U(1)$-equivariant pushout diagram concerning
spheres and the $2n$-ball:
\begin{equation}\label{d1}
\begin{gathered}
\xymatrix{
&S^{2n+1}&\\
S^{2n-1}\ar[ur]_{}&& B^{2n}\times S^1\ar[ul]^{}\\
&S^{2n-1}\times S^1~.\ar[lu]\ar[ru]&
}
\end{gathered}
\end{equation}
The  equivariance of the above diagram allows it to descend to the quotient spaces:
\begin{equation}
\begin{gathered}
\xymatrix{
&\mathbb{C}{\rm P}^n&\\
\mathbb{C}{\rm P}^{n-1}\ar[ur]_{}&& B^{2n}\phantom{X}\ar[ul]^{}\\
&S^{2n-1}~.\ar[lu]\ar[ru]&
}
\end{gathered}
\end{equation}
Thus we obtain a diagram that manifests the CW-complex structure of complex projective 
spaces~$\mathbb{C}{\rm P}^n$.

For $n=1$, a $q$-deformed version of the diagram~\eqref{d1} was considered in \cite{HW14}
and proved to be a $U(1)$-equivariant pullback diagram of C*-algebras. This pullback structure
provided therein an alternative way for computing an index pairing for noncommutative 
line bundles over the standard Podle\'s quantum sphere~\cite{Pod87,h-pm00}.
For $n=2$, an analogous result was called for in \cite{FHMZ17} in the context of the multipullback
noncommutative deformation of complex projective spaces. This led to obtaining the general result
for an arbitrary $n$ in this paper: a $U(1)$-equivariant pullback structure of the C*-algebra of the
Vaksman--Soibelman quantum sphere $S^{2n+1}_q$ is a prototype of the main theorem herein.
%\todo[inline]{of the main theorem herein --> of our main result.}

Our main theorem is based on a general concept of a trimmable graph~\cite{HKT18}:
a finite graph $E$ is $\bar{v}$-trimmable iff it consists
of a subgraph $E''$ emitting at least one edge to an external vertex $\bar{v}$ whose only outgoing
 edge $\bar{e}$ is a loop and such that
all edges other than  $\bar{e}$ that end in $\bar{v}$ begin in a vertex emitting an edge that ends
 not in~$\bar{v}$. A trimmable graph $E$ can be trimmed to its subgraph~$E''$. There is an intermediate
subgraph $E'$ of  $E$ that is a one-sink extension of~$E''$. 
Much as one
 defines a one-sink extension,
we can define a one-loop extension, so that a one-loop extension of any finite graph is automatically
a trimmable graph.

The paper is organized as follows.
To make it self-contained, we start by recalling some results on graph C*-algebras in the preliminaries. Then
we proceed to Section~2 where we prove the main result and 
study general K-theoretical benefits
of decomposing a trimmable graph C*-algebra into simpler building blocks.
The last section is devoted to four  examples of two different types.
 The first two of them are in the spirit of abstract
graph algebras (one-loop extensions of the Cuntz algebra $\mathcal{O}_2$ and the Toeplitz
 algerbra~$\mathcal{T}$), whereas the remaining two are in the spirit of noncommutative topology
($q$-deformations of spheres, balls, complex projective spaces, lens spaces, teardrops).

\subsection{Preliminaries}\label{sec:preliminaries}
\noindent

In this section, we recall some general results from the theory of graph C*-algebras. Our main references are
 \cite{aHR97,BPRS,R05}. 
We adopt the conventions of \cite{BPRS}, i.e., the roles of source and range map are exchanged with respect to~\cite{R05}.

Let $E$ be a directed graph, $E_0$ the set of vertices, $E_1$ the set of edges, \mbox{$s:E_1\to E_0$} and \mbox{$r:E_1\to E_0$} 
the source and 
range map respectively. 
A directed graph is called 
\emph{row-finite} if $s^{-1}(v)$ is a finite set for every $v\in E_0$.
It is called \emph{finite} if both sets $E_0$ and $E_1$ are finite.
A \emph{sink} is a vertex $v$ with no outgoing edges, that is $s^{-1}(v)=\{e\in E_1:s(e)=v\}=\emptyset$. 
By a \emph{path} $e$ of length $|e|=k\geq 1$ we mean a directed path, i.e.\ a sequence of edges $e:=e_1\ldots e_k$, 
with $r(e_i)=s(e_{i+1})$ for all $i=1,\ldots,k-1$. 
We view vertices as paths of length zero. We denote the set of all paths in $E$ by 
${\rm Path}(E)$. We extend the maps $r$ and $s$ to ${\rm Path}(E)$ by setting $s(e):=s(e_1)$ and $r(e):=r(e_k)$ for all $e$
of length $k\geq 1$, and $s(v):=v=:r(v)$ for all paths $v$ of length zero.
\begin{df}[Graph C*-algebra]
The \emph{graph \mbox{C*-algebra}} $C^*(E)$ of a row-finite graph $E$ is the universal \mbox{C*-algebra} generated by mutually 
orthogonal projections $\big\{P_v:v\in E_0\big\}$ and partial isometries $\big\{S_e:e\in E_1\big\}$ satisfying the \emph{Cuntz--Krieger relations}:
\begin{align}
S_e^*S_e &=P_{r(e)} && \text{for all }e\in E_1\text{, and}\tag{\text{CK1}} \label{eq:CK1} \\
\sum_{e\in E_1:\,s(e)=v}\!\! S_eS_e^*&=P_v && \text{for all }v\in E_0\text{ that are not sinks.}\tag{CK2} \label{eq:CK2}
\end{align}
The datum $\{S,P\}$ is called a \emph{Cuntz--Krieger $E$-family}.
\end{df}
\noindent
On the notational side, if $E'$ is a subgraph of $E$ and $\{S,P\}$ is the Cuntz--Krieger $E$-family,  
we will slightly abuse notation and denote the Cuntz--Krieger $E'$-family by $\{S,P\}$ as well. 

Any graph C*- algebra $C^*(E)$ can be endowed with a natural circle action 
$$\alpha:U(1)\longrightarrow\mathrm{Aut}(C^*(E)),$$ called the \emph{gauge action}. 
Using the universality of $C^*(E)$, it is defined
by being set on the generators:
$$
\alpha_\lambda(P_v)=P_v\;,\qquad
\alpha_\lambda(S_e)=\lambda S_e\;,\quad\text{where}\quad \lambda\in U(1),\quad v\in E_0,\quad e\in E_1.
$$
The fixed-point subalgebra under the gauge action is an AF-subalgebra of the form (e.g., see~\cite[Corollary~3.3]{R05}):
\begin{equation}\label{fixsub}
C^*(E)^{U(1)}=\overline{{\rm span}}\left\{S_x S^*_y:x,y\in {\rm Path}(E)\,,\;r(x)=r(y)\,,\;|x|=|y|\right\}.
\end{equation}
Here for a path $x:=x_1x_2\ldots x_n$ we set $S_x:=S_{x_1}S_{x_2}\ldots S_{x_n}$, and for a path $v$ of length $0$ we put 
$S_{v}:=P_{v}$.

The gauge action is a central ingredient in the gauge-invariant uniqueness theorem proved by an Huef and Raeburn 
\cite[Theorem~2.3]{aHR97} in the context of Cuntz--Krieger algebras~\cite{CK80}, 
and then generalized to graph C*-algebras of row-finite graphs
by Bates, Pask, Raeburn and Szyma\'nski \cite[Theorem~2.1]{BPRS}. This theorem, together with the universality of 
graph C*-algebras with respect to the Cuntz--Krieger relations, is an essential tool in proving that a given C*-algebra 
is isomorphic to a graph C*-algebra. We give here a slight reformulation of the result, more suitable for the purposes of this work.

\begin{thm}[Gauge-invariant uniqueness theorem \protect{\cite{aHR97}, \cite{BPRS}, \cite[Theorem~2.2]{R05}}]\label{gut}
Let $E$ be a row-finite graph with the Cuntz--Krieger family $\{S,P\}$, 
let $A$ be a C*-algebra with a continuous action of $U(1)$ and $\rho:C^*(E)\to A$ a $U(1)$-equivariant 
$*$-homomorphism. If $\rho(P_v)\neq 0$ for all $v\in E_0$, then $\rho$ is injective.
\end{thm}

To understand gauge-invariant ideals of graph C*-algebras, 
we need to introduce two kinds of subsets of the set of vertices.
Recall that given two vertices $v,w\in E_0$, whenever $w$ is 
\emph{reachable} from $v$, that is whenever there is a path from $v$ to $w$, we write 
$v>w$, and we write $v\geq w$ if $v>w$ or $v=w$. A subset $H$ of $E_0$
is called \emph{hereditary} iff $v \geq w$ and $v \in H$ imply $w \in H$.
A hereditary subset $H$ is \emph{saturated} iff every
vertex which feeds into $H$ and only into $H$ is again in~$H$.
We denote by $\overline{H}$ the \emph{saturation} of a hereditary subset~$H$,
that is the smallest saturated subset containing $H$.

It follows from \cite[Lemma~4.3]{BPRS} that, for any hereditary subset $H$,
the (algebraic) ideal 
generated by $\{P_v:v\in H\}$ is of the form
  \begin{equation}\label{herideal}
I_E(H)=\mathrm{span}\left\{S_xS_y^*:x,y\in\text{Path}(E)\,,\;r(x)=r(y)\in \overline{H}\right\}.
  \end{equation}
Equation (\ref{herideal}) will play an essential role in the proof of Theorem \ref{main}.

By \cite[Theorem 4.1 (a)]{BPRS}, given a row-finite graph $E$, the 
gauge-invariant ideals in the graph algebra $C^*(E)$ are in 
one-to-one correspondence with saturated hereditary subsets of $E_0$.
By \cite[Theorem~4.1~(b)]{BPRS}, quotients by (closed) ideals generated by saturated hereditary subsets can 
be realised also as graph C*-algebras 
by constructing a \emph{quotient graph}. Given a saturated hereditary subset $H$ of $E_0$, the quotient graph $E/H$ is the 
graph obtained
by removing from $E$ all the vertices in $H$ and all the edges whose range is in $H$, i.e. $(E/H)_0:=E_0\setminus H $ and 
$ (E/H)_1:=\{e\in E_1:r(e)\notin H\}$. 
As a consequence, we have a $U(1)$-equivariant 
isomorphism
\begin{equation}\label{quotient}
C^*(E)\slash \overline{I_{E}(H)}\cong C^*(E/H),
\end{equation}
where $ \overline{I_{E}(H)}$ is the norm closure of $I_E(H)$.
\section{Trimmable graph C*-algebras}

The following notion of a trimmable graph was introduced in~\cite[Definition~2.1]{HKT18}
in the context of Leavitt path algebras. 

\begin{df}[\cite{HKT18}]\label{trim} Let $E$ be a finite graph with a distinguished vertex $\bar{v}$
emitting a loop~$\bar{e}$.
A graph $E$ is called $\bar{v}$-\emph{trimmable} iff the pair $(E,\bar{v})$ satisfies the following conditions
\begin{gather}
s^{-1}(\bar{v})=\{\bar{e}\}, \quad r^{-1}(\bar{v})\setminus\{\bar{e}\}\neq\emptyset,\tag{T1}\label{eq:trim0}
\\
\forall\; v\in s\big(r^{-1}(\bar{v})\setminus\{\bar{e}\}\big)\,\colon\quad s^{-1}(v)\setminus r^{-1}(\bar{v})\neq\emptyset.\tag{T2}\label{eq:trim}
\end{gather}
We call $C^*(E)$ a $\bar{v}$-\emph{trimmable graph C*-algebra} iff $E$ is $\bar{v}$-trimmable.
\end{df}
\noindent
Note that conditions (\ref{eq:trim0}) and (\ref{eq:trim}) imply that $\{\bar{v}\}$ is a saturated hereditary subset of $E_0$.
Furthermore, if (\ref{eq:trim}) is not satisfied, i.e.\ for some $v\in s(r^{-1}(\bar{v})\setminus\{\bar{e}\})$ the set difference
$s^{-1}(v)\setminus r^{-1}(\bar{v})$ is empty, then the quotient map $C^*(E)\to C^*(E/\{\bar{v}\})$ would not
be well defined as it would map all elements $S_y$, where $y\in r^{-1}(\bar{v})$, to zero,
thus violating the Cuntz-Krieger relations for $C^*(E/\{\bar{v}\})$.

There is an ample supply of trimmable graphs because, given any finite graph $E''$, 
we can create a trimmable graph $E$ by 
taking a~\emph{one-loop extension}  of $E''$. We define one-loop extensions in the spirit of 
\emph{one-sink extensions} 
defined in \cite[Definition~1.1]{RTW04}.
\begin{df}
Let $E''$ be a finite graph. A finite graph $E$ is called a~\emph{one-loop extension} of $E''$ iff
the following conditions are satisfied:
\vspace{-2mm}
\begin{enumerate}
\item $E''$ is a subgraph of $E$,
\item $E_0\setminus E''_0=\{\bar{v}\}$ (there is only one vertex outside of $E''$),
\item $s^{-1}(\bar{v})=\{\bar{e}\}$ and $r(\bar{e})=\bar{v}$ (the only edge outgoing from $\bar{v}$ is a loop),
\item $r^{-1}(\bar{v})\setminus\{\bar{e}\}\neq\emptyset$ (there is at least edge connecting $E''$ with $\bar{v}$),
\item if $v$ is a sink in $E''$, then it remains a sink in $E$ (equivalent to the condition \eqref{eq:trim}).
\end{enumerate}
\end{df}
\noindent
Note that, for any trimmable graph $E$, there is an intermediate graph $E'$ that is a subgraph of $E$
and a one-sink extension~of~$E''$.

\subsection{A $K_1$-generator for trimmable graphs without sinks}
Given a $\bar{v}$-trimmable graph $E$, the Cuntz--Krieger relations imply that the partial isometry associated to the loop $\bar{e}$
based at $\bar{v}$ is a normal operator. This fact can be used to construct a~distinguished class in $K_1(C^*(E))$.

\begin{prop}
\label{prop:K_1}
 Let $E$ be a $\bar v$-trimmable graph without sinks. Then $K_1(C^*(E))$ contains a copy of $\Z$ generated by the class of the unitary 
 \[ U = S_{\bar{e}} + (1-S_{\bar{e}}S_{\bar{e}}^*). \]
\end{prop}

\begin{proof}
Let $v_i, i=1, \dots, n-1$, be the vertices of $E$ different from $\bar{v}$ (recall that $E$ is finite). 
By the trimmability conditions \eqref{eq:trim0} and \eqref{eq:trim} in Definition \ref{trim}, the incidence 
matrix for the graph 
$E$ is an $n \times n$ matrix of the form
\[
A_E= \left( \begin{array}{c|ccccc}
    1 & 0 & 0 & \dots & 0 & 0  \\ \hline
    A(1,0) & & & & & \\
    \vdots & & & A_{E''} & &\\
    A(n-1,0) & & & &  &
   \end{array} \right).
\]
Here $\bar{v}$ vertex is listed first, $A_{E''}$ is the incidence matrix of the quotient graph $E'':=E/\{\bar{v}\}$
and $A(i,0) := \# \lbrace e \in E^{1} \ : \ r(e)=\bar{v},~s(e)=v_i \rbrace.$ The actual values of the $A(i,0)$ do not matter for our proof. 
The above implies that the first column of the matrix $1-A^t_E$ contains only zeros. This means that the vector $(1,0,\dots,0)$ generates a copy of $\Z$ inside $\ker (1-A^t_E)$. Since the graph $E$ has no sinks, the C*-algebra $C^*(E)$ is isomorphic to the 
Cuntz--Krieger algebra of the edge matrix $A_E$ of the graph, so $\ker (1-A^t_E)\cong K_1(C^*(E))$ by \cite[Proposition~3.1]{Cu81}.

Next, let us consider the partial isometry $S_{\bar{e}}$ associated to the loop $\bar{e}$ based at the 
distiguished vertex~$\bar{v}$. 
Since $\bar{v}$ does not emit any 
other edge besides the loop $\bar{e}$, we deduce from the Cuntz--Krieger relations that $S_{\bar{e}}$ is a normal element. 
Then one readily checks that the element $U := S_{\bar{e}} + (1-S_{\bar{e}}S_{\bar{e}}^*)$ is unitary.
The argument outlined in \cite[Section~2]{Roe95} allows us to conclude that $U$ generates a copy of $\Z$ inside $K_1(C^*(E))$ 
corresponding to the vector $(1,0,\dots,0)$.
\end{proof}

\subsection{An equivariant  pullback structure}
\label{sec:equiv}
In this section, we prove that every $\bar{v}$-trimmable graph 
C*-algebra $C^*(E)$ 
is $U(1)$-equivariantly isomorphic 
to the pullback
C*-algebra of $C^*(E')\otimes C(S^1)$ and $C^*(E'')$  over $C^*(E'')\otimes C(S^1)$, 
where $E'$ is the subgraph of $E$ obtained by removing the loop~$\bar e$ and $E'':=E/\{\bar v\}$. 
For this statement, we need to introduce $U(1)$-equivariant *-homomorphisms which give the aforementioned 
pullback structure.

We begin with the map that dualizes the gauge action $\alpha$, namely the gauge coaction 
$$
\delta:C^*(E)\longrightarrow C^*(E)\otimes C(\bS^1),
$$
given on generators by
$$
\delta(S_e)=S_e\otimes u\;,
\qquad
\delta(P_v)=P_v\otimes 1\;,\quad
\text{for~all}\quad e\in E_1 \quad\text{and}\quad v\in E_0.
$$
Here we denote by $u$ the standard unitary generator of $C(\bS^1)$ and by $\{S,P\}$ the Cuntz--Krieger $E$-family.
The gauge coaction is $U(1)$-equivariant with respect to the gauge action on $C^*(E)$ and the action on
$C^*(E)\otimes C(S^1)$ given by the standard action on the rightmost tensorand.

The singleton set $\{\bar{v}\}$ is a saturated hereditary subset of both $E_0$ and $E'_0$ 
(note that $E/\{\bar{v}\}=E'/\{\bar{v}\}$) and,
by the one-to-one correspondence of gauge-invariant ideals and saturated hereditary subsets,
the quotient maps 
$$
\pi_1:C^*(E)\longrightarrow C^*(E''),\qquad
\pi_2:C^*(E')\longrightarrow C^*(E''),
$$
are $U(1)$-equivariant with respect to the gauge actions on $C^*(E)$, $C^*(E')$ and $C^*(E'')$.
In the forthcoming theorem, we will consider a *-homomorphism
\[
\pi_2\otimes\id :C^*(E')\otimes C(S^1)\longrightarrow C^*(E'')\otimes C(S^1)
\]
viewed as a $U(1)$-equivariant map with respect to the standard $U(1)$-action on~$C(S^1)$.

Finally, using the condition (\ref{eq:trim0}), one readily verifies that  
the assignment
\[ f(P_v) = P_v\otimes 1,\quad v\in E_0,
\qquad\quad
f(S_e) = 
\begin{cases}
  P_{\bar{v}}\otimes u & \text{ if } e = \bar{e},\\
  S_e\otimes u & \text{ if } e\in E'_1\,,
\end{cases}
\]
defines a map
\[
f\colon C^*(E)\longrightarrow C^*(E')\otimes C(S^1),
\]
as it preserves the relations 
\eqref{eq:CK1}-\eqref{eq:CK2} of the graph C*-algebra $C^*(E)$. 
It is equally straightforward to check that $f$ is equivariant with respect to the gauge action on $C^*(E)$
and the action on $C^*(E')\otimes C(S^1)$ given by the standard action on the rightmost tensorand.

\begin{thm}\label{main}
Let $E$ be a $\bar{v}$-trimmable graph, $E'$ the subgraph of $E$ obtained  
by removing the  unique outgoing loop~$\bar{e}$, and $E''$ the subgraph of $E'$
obtained  by removing the vertex $\bar{v}$ and all edges ending in~$\bar{v}$.
Then the following diagram of the above-defined $U(1)$-equivariant *-homomorphisms
{\begin{equation}
\begin{gathered}
\label{eq:pbdiag}
\xymatrix{
&C^*(E)\ar[ld]_{\pi_1}\ar[rd]^{f}&\\
C^*(E'')\ar[dr]_{\delta\quad}&& 
C^*(E')\otimes C(S^1)
\ar[dl]^{\quad\pi_2\otimes\id}\\
&C^*(E'')\otimes C(S^1)&
}
\end{gathered}
\end{equation}}
is a pullback diagram of $U(1)$-C*-algebras. 
\end{thm}

\begin{proof}
The above diagram is clearly commutative, 
so there is a *-homomorphism $F$ mapping $C^*(E)$  into the pullback C*-algebra.
Injectivity of $F$ follows from injectivity of $f$, 
which is a consequence of the gauge-invariant uniqueness theorem (Theorem~\ref{gut}).
Furthermore, since $\pi_1$ and $\pi_2\otimes \id$ are surjective, using a C*-algebraic incarnation \cite[Proposition~3.1]{P99}
of a well-known characterization of when a commutative diagram is a pullback diagram (e.g. see~\cite[Lemma~4.1]{HKT18}), 
to prove the surjectivity of $F$, it only remains to check whether
\begin{equation}\label{kernel}
{\rm ker}(\pi_2\otimes\id)\subseteq f({\rm ker}(\pi_1)).
\end{equation}

Again, we observe that $\{\bar{v}\}$ is a saturated hereditary subset of both $E_0$ and $E'_0$. 
Therefore, it generates gauge-invariant ideals $\overline{I_{E}(\bar{v})}$ 
and $\overline{I_{E'}(\bar{v})}$
in $C^*(E)$ and $C^*(E')$ respectively. 
It follows from \eqref{herideal} and \eqref{quotient} that
\begin{equation*}
\ker(\pi_1)=\overline{I_E(\bar{v})}\qquad\text{and}\qquad
\ker(\pi_2)=\overline{I_{E'}(\bar{v})},
\end{equation*}
so every element in $I_{E'}(\bar{v})\otimes \mathbb{C}[u,u^{-1}]$ is a linear combination of elements of the form
\[
\left(\sum_{i=1}^n k_iS_{x_i}S^*_{y_i}\right) \otimes u^m,
\quad k_i\in\mathbb{C},\quad x_i,y_i\in{\rm Path}(E'),\quad r(x_i)=r(y_i)=\bar{v},\quad m\in\mathbb{Z}.
\]
Given an element as above, we can show that it belongs to $f(I_E(\bar{v}))$ using the following equality
\[
\sum_{i=1}^n f(S_{x_i}S_{\bar{e}}^{m-(|x_i|-|y_i|)}S^*_{y_i})=\left(\sum_{i=1}^n k_iS_{x_i}S^*_{y_i}\right) \otimes u^m.
\]
Hence $I_{E'}(\bar{v})\otimes\mathbb{C}[u,u^{-1}]\subseteq f(I_E(\bar{v}))$.
Finally, to prove \eqref{kernel}, we compute
\[
\ker(\pi_2\otimes\id)=\ker(\pi_2)\otimes C(S^1)=\overline{I_{E'}(\bar{v})\otimes \mathbb{C}[u,u^{-1}]}\subseteq \overline{f(I_E(\bar{v}))}=f(\ker(\pi_1)).
\]
Here the last equality follows from the fact that the image of a C*-algebra under 
any *-homomorphism is closed.
\end{proof}
As a corollary, by equivariance of all maps in diagram (\ref{eq:pbdiag}) and compactness of $U(1)$, 
we obtain a pullback diagram at 
the level of fixed-point subalgebras:
\begin{cor}\label{main_fixed}
The diagram of *-homomorphisms 
\begin{equation}\label{eq:pbfixed}
\begin{gathered}
\xymatrix{
&C^*(E)^{U(1)}\ar[ld]_{\widetilde{\pi_1}}\ar[rd]^{\widetilde{f}}&\\
C^*(E\rq{}')^{U(1)}\ar[dr]_{\iota\quad}&& 
C^*(E')
\ar[dl]^{\quad\pi_2}\\
&C^*(E\rq{}\rq{})&
}
\end{gathered}
\end{equation}
is a pullback diagram of C*-algebras. Here $\widetilde{\pi_1}$ and $\widetilde{f}$ denote *-homomorphisms 
that are restrictions-corestrictions to the
fixed-point subalgebras of $\pi_1$ and $f$ respectively, and $\iota$ is the subalgebra inclusion.
\end{cor}

\subsection{Mayer--Vietoris exact sequences in K-theory}
\label{ss:MV}

Any pullback diagram of C*-algebras induces a six-term exact sequence in K-theory that goes under the name of 
Mayer--Vietoris exact sequence (see for example \cite[Section~1.3]{BHMS05}, \cite[Section~1.2.3]{BM} and 
\cite[Theorem~21.5.1]{B98}). 
In this subsection, we describe the Mayer--Vietoris exact sequence for trimmable graph C*-algebras 
and their gauge-invariant subalgebras.

Let $E$ be a $\bar{v}$-trimmable graph. The Mayer--Vietoris six-term exact sequence associated to 
the diagram \eqref{eq:pbdiag} reads
{\footnotesize
\begin{equation}\label{6term}
\xymatrixcolsep{3.5pc}
\xymatrix{
K_0\big(C^*(E)\big)\ar[r]^-{\left({\pi_1}_{*}, {f}_{*}\right)} & K_0\big(C^*(E'')\big)\!\oplus\! K_0\big(C^*(E')\!\otimes\! C(S^1)\big) \ar[r]^-{\delta_{*}-(\pi_2 \otimes \id)_*}  & K_0\big(C^*(E'')\!\otimes\! 
C(S^1)\big)\ar[d]^{\partial_{01}}\\ 
K_1\big(C^*(E'')\!\otimes\!C(S^1)\big)\ar[u]^{\partial_{10}} 
& K_1\big(C^*(E'')\big)\!\oplus\! K_1\big(C^*(E')\!\otimes\! C(S^1)\big) \ar[l]^-{\delta_{*}-
(\pi_2 \otimes \id)_*} & K_1\big(C^*(E)\big). \ar[l]^-{\left({\pi_1}_{*}, {f}_{*}\right)}
}
\end{equation}
}

\noindent
Here $\partial_{10}$ is a Milnor connecting homomorphism and $\partial_{01}$ 
is a Bott connecting homomorphism. Furthermore,
the pullback diagram (\ref{eq:pbfixed}) of fixed-point subalgebras leads to another 
six-term exact sequence in K-theory:
\begin{equation}\label{6termfixed}
\begin{gathered}
\xymatrixcolsep{2.5pc}
\xymatrix{
K_0(C^*(E)^{U(1)}) \ar[r]^-{\left(\widetilde{\pi_1}_{*}, \widetilde{f}_{*}\right)}
& K_0(C^*(E'')^{U(1)})\oplus K_0(C^*(E')) \ar[r]^-{{\pi_2}_{*}-\iota_{*}}
& K_0(C^*(E'')) \ar[d]^{\partial_{01}}\\
K_1(C^*(E''))\ar[u]^{\partial_{10}}
& K_1(C^*(E'')^{U(1)}) \oplus K_1(C^*(E')) \ar[l]^-{{\pi_2}_{*}-\iota_{*}}
& K_1(C^*(E)^{U(1)}). \ar[l]^-{\left(\widetilde{\pi_1}_{*}, \widetilde{f}_{*}\right)}}
\end{gathered}
\end{equation}
Next, since gauge-invariant subalgebras of  graph C*-algebras are always~AF-algebras,
 their odd-K-groups vanish, so we obtain a simpler six-term exact sequence:
\begin{equation}\label{af6term}
\begin{gathered}
\xymatrixcolsep{2.5pc}
\xymatrix{
K_0(C^*(E)^{U(1)}) \ar[r]^-{\left(\widetilde{\pi_1}_{*}, \widetilde{f}_{*}\right)}
& K_0(C^*(E'')^{U(1)})\oplus K_0(C^*(E')) \ar[r]^-{{\pi_2}_{*}-\iota_{*}}
& K_0(C^*(E'')) \ar[d]\\
K_1(C^*(E''))\ar[u]^{\partial_{10}}
& K_1(C^*(E')) \ar[l]^-{{\pi_2}_{*}-\iota_{*}}
& 0. \ar[l]}
\end{gathered}
\end{equation}
\section{Examples}\label{sec:applications}

\subsection{A one-loop extension of the Cuntz algebra $\mathcal{O}_2$}

Recall that the Cuntz algebra $\mathcal{O}_2$ can be viewed as the graph C*-algebra of the graph $\Lambda''$
consisting of two loops starting at the unique vertex $w$.
Let us now consider the one-loop extension $\Lambda$ of this graph obtained by adding one
outgoing edge from $w$ to $\bar{v}$ (see Figure~\ref{fig:cuntz}).

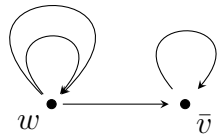
\begin{figure}[h]
\begin{center}
\begin{tikzpicture}[>=stealth,node distance=50pt,main node/.style={circle,inner sep=2pt},
freccia/.style={->,shorten >=2pt,shorten <=2pt},
ciclo/.style={out=130, in=50, loop, distance=40pt, ->},
circle/.style={out=135, in=45, loop, distance=65pt, ->}]

      \node[main node] (1) {};
      \node (2) [right of=1] {};

      \filldraw (1) circle (0.06) node[below left] {$w$};
      \filldraw (2) circle (0.06) node[below right] {$\bar{v}$};

      \path[freccia] (1) edge[ciclo] (1);
      		\path[freccia] (1) edge[circle] (1);
			\path[freccia] (2) edge[ciclo] (2);

      \path[freccia] (1) edge (2);
			   
\end{tikzpicture}
\end{center}
\vspace{-10pt}
\caption{The graph $\Lambda$.}
\label{fig:cuntz}
\end{figure}
\noindent 
Denote by $\Lambda'$ the one-sink extension of $\Lambda''$ obtained by adding one
outgoing edge from $w$ to $\bar{v}$. 
By~Theorem~\ref{main},
we have the following $U(1)$-equivariant pullback structure:
\begin{equation}\label{cuntz}
\begin{gathered}
\xymatrix{
& C^*(\Lambda) \ar[ld] \ar[rd] \\
 \mathcal{O}_2 \ar[rd]
& 
& C^*(\Lambda')\otimes C(S^1). \ar[ld]\\
& \mathcal{O}_2\otimes C(S^1)
}
\end{gathered}
\end{equation}

The fixed-point subalgebra $\mathcal{O}_2^{U(1)}$ is isomorphic to the CAR algebra 
$M_{2^\infty}(\C)$ \cite[\textsection 1.5]{Cu77}. 
Hence, by~Corollary~\ref{main_fixed}, we have another pullback diagram:
\begin{equation}\label{car}
\begin{gathered}
\xymatrix{
& C^*(\Lambda)^{U(1)} \ar[ld] \ar[rd] \\
 M_{2^{\infty}}(\C) \ar[rd]
& 
& C^*(\Lambda'). \ar[ld]\\
& \mathcal{O}_2
}
\end{gathered}
\end{equation}
The diagram~\eqref{af6term} for this example reads
\begin{equation}\label{6cuntz}
\begin{gathered}
\xymatrixcolsep{2.5pc}
\xymatrix{
K_0(C^*(\Lambda)^{U(1)}) \ar[r]
& K_0(M_{2^{\infty}}(\C))\oplus K_0(C^*(\Lambda')) \ar[r]
& K_0(\mathcal{O}_2) \ar[d]\\
K_1(\mathcal{O}_2)\ar[u]
& K_1(C^*(\Lambda')) \ar[l]
& 0. \ar[l]}
\end{gathered}
\end{equation}
Now, to compute $K_0(C^*(\Lambda)^{U(1)})$, we observe that:
\begin{enumerate}
\item $K_0(\mathcal{O}_2)=0=K_1(\mathcal{O}_2)$ \cite{Cu78}, \cite[Theorem~3.7, Theorem~3.8]{Cu81K},
\item $K_0(M_{2^{\infty}}(\C))\cong\Z[\frac{1}{2}]$, 
where $\Z[\frac{1}{2}]$ is the group of~dyadic rationals (e.g., see \cite[Example~IV.3.4]{Dav96}),
\item $K_0(C^*(\Lambda'))=\Z[P_{\bar{v}}]$ and $K_1(C^*(\Lambda'))=0$,
where $P_{\bar{v}}$ is the vertex projection of $\bar{v}$.
\end{enumerate}
Here the last statement follows from \cite[Lemma~5.2]{RTW04} because $\Lambda'$ is a one-sink extension~of~$\Lambda''$.
Hence, from the diagram \eqref{6cuntz}, we conclude:
\begin{prop}
Let $C^*(\Lambda)$ be the graph C*-algebra of the graph given by Figure~\ref{fig:cuntz}. 
The~gauge-invariant subalgebra $C^*(\Lambda)^{U(1)}$ has the~following even-K-group: 
\[
K_0(C^*(\Lambda)^{U(1)})\cong K_0(M_{2^{\infty}}(\C))\oplus K_0(C^*(\Lambda'))\cong\Z[\frac{1}{2}]\oplus \Z[P_{\bar{v}}].
\]
\end{prop}
The above result agrees with \cite[Proposition~4.1.2]{PR96}, where the K-theory of fixed-point subalgebras 
of arbitrary Cuntz--Krieger algebras (note that $\Lambda$ is finite and without sinks) was computed using a dual 
Pimsner--Voiculescu sequence.

\subsection{A one-loop extension of the Toeplitz algebra $\mathcal{T}$}\label{toe}
Let us now consider an~example that goes beyond~\cite[Proposition~4.1.2]{PR96} and connects with Section~\ref{teardrops}. 
To this end, we take the standard graph $Q_2''$ encoding the Toeplitz algebra~$\mathcal{T}$, that is the graph consisting
of two vertices $v_0^0$ and $v^1_0$, one loop $e^0_0$ based at $v_0^0$, and one edge $e^{01}_0$ from $v_0^0$ to~$v_0^1$.
Note that $v_0^1$ is a sink, so $\mathcal{T}$ is not a Cuntz--Krieger algebra. 
Next, we consider the one-loop extension $Q_2$ of this graph obtained by adding one
outgoing edge from $v^0_0$ to $v^1_1$ (see Figure~\ref{fig:cuntz2}).
Denote by $Q_2'$ the one-sink extension of $Q_2''$ obtained by adding one
outgoing edge from $v^0_0$ to $v^1_1$. The C*-algebra of the graph $Q_2'$
is isomorphic to the C*-algebra of the equatorial Podle\'s quantum sphere $S^2_{q\infty}$ \cite{Pod87, HS02}.

\begin{figure}[h]
\begin{center}
\begin{tikzpicture}[>=stealth,node distance=40pt,main node/.style={circle,inner sep=2pt},
freccia/.style={->,shorten >=1pt,shorten <=1pt},
ciclo/.style={out=130, in=50, loop, distance=40pt, ->},
ciclo2/.style={out=235, in=315, loop, distance=40pt, ->}]

      \node[main node] (1) {};
      \node (2) [below of=1] {};
      \node (3) [left of=2]{};
      \node (4) [right of=2]{};

      \filldraw (1) circle (0.06) node[right] {$\ v_0^0$};
      \filldraw (3) circle (0.06) node[left] {$v_0^1$};
      \filldraw (4) circle (0.06) node[right] {$v_1^1$};

			\path[freccia] (1) edge[ciclo] (1);
			\path[freccia] (3) edge[ciclo2] (3);

      \path[freccia] (1) edge (4);         							
\path[freccia] (1) edge (3)	; 
\end{tikzpicture}
\end{center}
\vspace{-10pt}
\caption{The graph $Q_2$.}
\label{fig:toeplitz}
\end{figure}
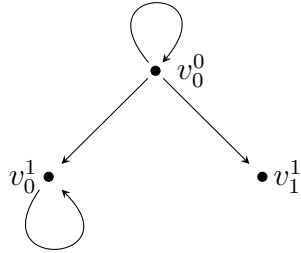
\noindent
By Theorem~\ref{main},
we have the following $U(1)$-equivariant pullback structure:
\begin{equation}\label{cuntz2}
\begin{gathered}
\xymatrix{
& C^*(Q_2) \ar[ld] \ar[rd] \\
 \mathcal{T} \ar[rd]
& 
& C(S^2_{q\infty})\otimes C(S^1). \ar[ld]\\
& \mathcal{T}\otimes C(S^1)
}
\end{gathered}
\end{equation}

To compute the fixed-point subalgebra $\mathcal{T}^{U(1)}$, we take advantage of~\eqref{fixsub} and combine it
with the fact that $\mathcal{T}$ is the unital universal C*-algebra generated by a single isometry~$s$~\cite{c-la66}.
Indeed, identifying the isometry $s$ with the sum of partial isometries $S_{e^0_0}+S_{e^{01}_0}$, one easily computes
that $\mathcal{T}^{U(1)}=\overline{{\rm span}}\left\{s^k(s^*)^k:k\in\mathbb{N} \right\}$, where $s^0(s^*)^0=1$.
Hence $\mathcal{T}^{U(1)}$
is a commutative AF-algebra generated by countably many
orthogonal projections, so it is isomorphic to 
the C*-algebra of continuous complex-valued functions on the Cantor set $X$. 
Thus, by Corollary~\ref{main_fixed}, we have another pullback diagram:
\begin{equation}\label{car2}
\begin{gathered}
\xymatrix{
& C^*(Q_2)^{U(1)} \ar[ld] \ar[rd] \\
C(X) \ar[rd]
& 
& C(S^2_{q\infty}). \ar[ld]\\
& \mathcal{T}
}
\end{gathered}
\end{equation}
The diagram~\eqref{af6term} for this example reads
\begin{equation}
\xymatrixcolsep{2.5pc}
\xymatrix{
K_0(C^*(Q_2)^{U(1)}) \ar[r]
& K_0(C(X))\oplus K_0(C(S^2_{q\infty})) \ar[r]
& K_0(\mathcal{T}) \ar[d]\\
K_1(\mathcal{T})\ar[u]
& K_1(C(S^2_{q\infty})) \ar[l]
& 0. \ar[l]}
\end{equation}
The K-groups of all the algebras involved here except for $C^*(Q_2)^{U(1)}$ are as follows:
\begin{enumerate}
\item $K_0(\mathcal{T})=\Z[1]$ and $K_1(\mathcal{T})=0$, 
\item $K_0(C(X))\cong\bigoplus_{\mathbb{N}}\Z$, 
\item $K_0(C(S^2_{q\infty}))\cong\Z^2$ and $K_1(C(S^2_{q\infty}))=0$.
\end{enumerate}
Hence there is a short exact sequence
\begin{equation}
0\longrightarrow K_0(C^*(Q_2)^{U(1)})\longrightarrow \bigoplus_{\mathbb{N}}\Z 
\longrightarrow \Z[1] \longrightarrow 0,
\end{equation}
so $K_0(C^*(Q_2)^{U(1)})$ is a countably-generated subgroup of a free abelian group. Thus we have proved:
\begin{prop}\label{toepro}
Let $C^*(Q_2)$ be the graph C*-algebra of the graph given by Figure~\ref{fig:toeplitz}. 
The~gauge-invariant subalgebra $C^*(Q_2)^{U(1)}$
has the following even-K-group: 
\[
K_0(C^*(Q_2)^{U(1)})\cong\bigoplus_{\mathbb{N}}\Z.
\]
\end{prop}

To end with, let us observe that,  since any AF-algebra is Morita equivalent to a~graph C*-algebra by \cite[Theorem~1]{Dr00},
in principle, one could compute $K_0(C^*(Q_2)^{U(1)})$ using~\cite{Dr00}. However, such a calculation might still be difficult,
and our method reduces the calculation of the even K-group of the gauge-invariant subalgebra of a graph algebra
 to an analogous  calculation for a
\emph{simpler}  graph algebra. In this particular case, the latter calculation is immediate.

\subsection{The Vaksman--Soibelman quantum spheres and projective spaces}

In 1991, Vaksman and Soibelman \cite{VS91} defined a class of odd-dimensional quantum spheres 
$\bS^{2n+1}_q$, where $n$ is a non-negative integer. 
Their C*-algebras can be viewed as $q$-deformations of the C*-algebras of continuous functions on odd-dimensional spheres 
$\bS^{2n+1}$, 
where  $q\in[0,1]$ is a deformation parameter. A~decade later, Hong and Szyma\'nski \cite{HS02} 
showed that in any dimension and for $q\in[0,1)$ these spheres can be realised as graph C*-algebras. 
In the same paper, they define even-dimensional
noncommutative balls $C(\bB^{2n}_q)$ using noncommutative double supspension, and in \cite{HS08} they give their
graph C*-algebraic presentation. 

\subsubsection{Spheres}

The C*-algebra $C(\bS^{2n+1}_q)$ of the $(2n+1)$-dimensional 
quantum sphere  is isomorphic, for any $q\in[0,1)$, to the graph C*-algebra of the graph 
$L_{2n+1}$ \cite[Theorem~4.4]{HS02} (see Figure \ref{fig:sphere}) with
\vspace*{-2mm}
\begin{list}{{\tiny\raisebox{1pt}{$\blacksquare\;$}}}{\leftmargin=2em \itemsep=2pt}
\item $n+1$ vertices $\{v_0,v_1,\ldots,v_n\}$,
\item one edge $e_{i,j}$ from $v_i$ to $v_j$ for all $0\leq i < j\leq n$,
\item one loop $e_i$ over each vertex $v_i$ for all $0\leq i\leq n$.
\end{list}

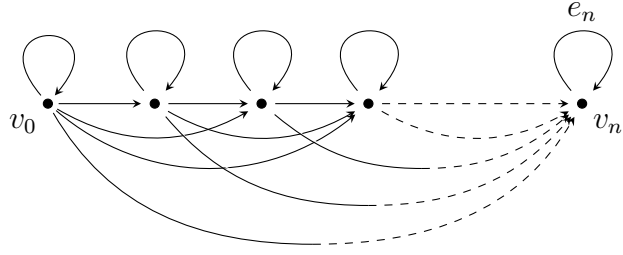
\begin{figure}[h]
\begin{center}
\begin{tikzpicture}[>=stealth,node distance=40pt,main node/.style={circle,inner sep=2pt},
freccia/.style={->,shorten >=1pt,shorten <=1pt},
ciclo/.style={out=130, in=50, loop, distance=40pt, ->}]

      \node[main node] (1) {};
      \node (2) [right of=1] {};
      \node (3) [right of=2] {};
      \node (4) [right of=3] {};
      \node (5) [right of=4] {};
      \node (6) [right of=5] {};

      \filldraw (1) circle (0.06) node[below left] {$v_0$};
      \filldraw (2) circle (0.06);
      \filldraw (3) circle (0.06);
      \filldraw (4) circle (0.06);
      \filldraw (6) circle (0.06) node[below right] {$v_n$};

      \path[freccia] (1) edge[ciclo] (1);
			\path[freccia] (2) edge[ciclo] (2);
			\path[freccia] (3) edge[ciclo] (3);
			\path[freccia] (4) edge[ciclo] (4);
      \path[freccia] (6) edge[ciclo] node[above] {$e_n$} (6);

      \path[freccia] (1) edge (2)
			               (2) edge (3)
										 (3) edge (4);
      \path[freccia,dashed] (4) edge (6);

      \path[freccia] (1) edge[bend right] (3)
			               (1) edge[bend right=40] (4)
										 (2) edge[bend right] (4);

      \path[white] (1) edge[bend right=60] coordinate (7) (6)
			               (2) edge[bend right=50] coordinate (8) (6)
										 (3) edge[bend right=40] coordinate (9) (6);

      \path[shorten <=1pt] (1) edge[out=-60,in=180] (7)
			               (2) edge[out=-50,in=180] (8)
										 (3) edge[out=-40,in=180] (9);

      \path[->,dashed,shorten >=1pt] (7) edge[out=0,in=240] (6)
			               (8) edge[out=0,in=230] (6)
										 (9) edge[out=0,in=220] (6);
										
			\path[freccia,dashed] (4) edge[bend right,dashed] (6);
\end{tikzpicture}
\end{center}
\vspace{-20pt}
\caption{The graph $L_{2n+1}$ of $C(S^{2n+1}_q)$.}
\label{fig:sphere}
\end{figure}
\noindent
Throughout this subsection, we denote the Cuntz--Krieger $L_{2n+1}$-family 
by $\{S,P\}$. To simplify notation, we set $S_{e_{i,j}}:=S_{i,j}$, $S_{e_j}:=S_j$ and $P_{v_j}:=P_j$, where $0\leq i<j\leq n$.

The C*-algebra $C(\bB^{2n}_q)$ of the Hong--Szyma\'nski $2n$-dimensional 
quantum ball \cite{HS08, HS06}  can be viewed as the graph C*-algebra of the graph $\Gamma_{2n}$ obtained from 
$L_{2n+1}$ by removing the loop~$e_n$ (see Figure~\ref{fig:ball}).
\begin{figure}[h]
\begin{center}
\begin{tikzpicture}[>=stealth,node distance=40pt,main node/.style={circle,inner sep=2pt},
freccia/.style={->,shorten >=1pt,shorten <=1pt},
ciclo/.style={out=130, in=50, loop, distance=40pt, ->}]

      \node[main node] (1) {};
      \node (2) [right of=1] {};
      \node (3) [right of=2] {};
      \node (4) [right of=3] {};
      \node (5) [right of=4] {};
      \node (6) [right of=5] {};

      \filldraw (1) circle (0.06) node[below left] {$v_0$};
      \filldraw (2) circle (0.06);
      \filldraw (3) circle (0.06);
      \filldraw (4) circle (0.06);
      \filldraw (6) circle (0.06) 
      node[below right] {$v_n$};

      \path[freccia] (1) edge[ciclo] (1);
			\path[freccia] (2) edge[ciclo] (2);
			\path[freccia] (3) edge[ciclo] (3);
			\path[freccia] (4) edge[ciclo] (4);

      \path[freccia] (1) edge (2)
			               (2) edge (3)
										 (3) edge (4);
      \path[freccia,dashed] (4) edge (6);

      \path[freccia] (1) edge[bend right] (3)
			               (1) edge[bend right=40] (4)
										 (2) edge[bend right] (4);

      \path[white] (1) edge[bend right=60] coordinate (7) (6)
			               (2) edge[bend right=50] coordinate (8) (6)
										 (3) edge[bend right=40] coordinate (9) (6);

      \path[shorten <=1pt] (1) edge[out=-60,in=180] (7)
			               (2) edge[out=-50,in=180] (8)
										 (3) edge[out=-40,in=180] (9);

      \path[->,dashed,shorten >=1pt] (7) edge[out=0,in=240] (6)
			               (8) edge[out=0,in=230] (6)
										 (9) edge[out=0,in=220] (6);
										
			\path[freccia,dashed] (4) edge[bend right,dashed] (6);

\end{tikzpicture}
\end{center}
\vspace{-20pt}
\caption{The graph $\Gamma_{2n}$ of $C(B^{2n}_q)$.}
\label{fig:ball}
\end{figure}
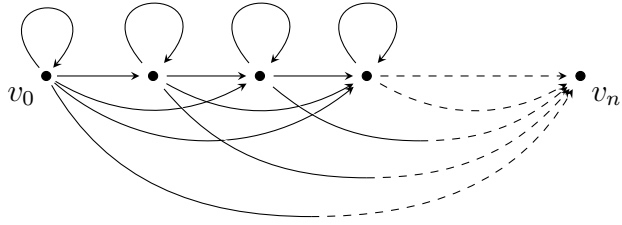

Since the graph $L_{2n+1}$ giving the Vaksman--Soibelman $(2n+1)$-sphere is $v_n$-trimmable, we immediately conclude from Theorem~\ref{main}
that the  diagram
\begin{equation}\label{vspull}
\begin{gathered}
\xymatrix{
& C(S^{2n+1}_q) \ar[ld]_{\pi_1^n} \ar[rd]^{f^n} \\
 C(S^{2n-1}_q) \ar[rd]^{\delta^n}
& 
& C(B^{2n}_q)\otimes C(S^1) \ar[ld]_{\pi_2^n\otimes{\id}}\\
& C(S^{2n-1}_q)\otimes C(S^1)
}
\end{gathered}
\end{equation}
is a pullback diagram of $U(1)$-C*-algebras. 
All the maps in the above diagram are special cases of the maps in the diagram \eqref{eq:pbdiag}.

In this example, the six-term exact sequence in K-theory given by the diagram \eqref{6term} reads
{\footnotesize
\begin{equation*}
%\xymatrixcolsep{3pc}
\xymatrix{
K_0\!\left(C(S_q^{2n+1})\right)\!\ar[r]^-{\left({\pi^n_1}_{*}, {f^n}_{*}\right)}& 
\!K_0\big(C(S_q^{2n-1})\big)\!\oplus\! K_0\big(C(B_q^{2n})\!\otimes\!C(S^1)\big)
\ar[r]^-{\delta^n_*\!-\!(\!\pi^n_2 \otimes \id)_*}\!
&K_0\big(C(S_q^{2n-1})\!\otimes\!C(S^1)\big)\ar[d]^{\partial_{01}}\\
K_1\big(C(S_q^{2n-1})\otimes C(S^1)\big)\ar[u]^{\partial_{10}} 
& K_1\big(C(S_q^{2n-1})\big)\!\oplus\! K_1\big(C(B_q^{2n})\!\otimes\! C(S^1)\big) \ar[l]^-{\ \delta^n_*\!-\!(\!\pi^n_2 \otimes \id)_*} 
&
K_1\big(C(S_q^{2n+1})\big).\ar[l]^-{\left({\pi^n_1}_{*}, {f^n}_{*}\right)}\\
}
\end{equation*}
}

\noindent
This diagram allows us to compute inductively an explicit formula of $K_1$-generators of~the~quantum odd spheres,
which then could be compared with results in \cite[Section~4.3]{HL04} and \cite{CET12}.

\subsubsection{Projective spaces}

We define the C*-algebra $C(\mathbb{C}\mathrm{P}^n_q)$ of the Vaksman--Soibelman quantum complex projective space 
$\mathbb{C}\mathrm{P}^n_q$ \cite{VS91} as the fixed-point subalgebra of
 $C(\bS^{2n+1}_q)$ under the gauge action of~$U(1)$. It can be viewed as the graph C*-algebra of the graph $M_n$
 consisting of the same vertices $v_0$, ..., $v_n$ as in the graph $L_{2n+1}$ (vertex projections are gauge invariant), with no 
 loops, and countably many edges between all pairs of vertices~\cite{HS02}. (See Figure~\ref{cpn}.)
 %, where thick arrows denote  infinitely many edges.)

\begin{figure}[h]
\begin{center}
\hspace*{-5mm}
\begin{tikzpicture}[>=stealth,node distance=1.8cm,main node/.style={circle,inner sep=2pt},
freccia/.style={->,shorten >=1pt,shorten <=1pt},
ciclo/.style={out=130, in=50, loop, distance=2cm, ->}]

      \node[main node] (1) {};
      \node (2) [right of=1] {};
      \node (3) [right of=2] {};
      \node (4) [right of=3] {};
      \node (5) [right of=4] {};
      \node (6) [right of=5] {};

      \filldraw (1) circle (0.06) node[below left] {$v_0$};
      \filldraw (2) circle (0.06);
      \filldraw (3) circle (0.06);
      \filldraw (4) circle (0.06);
      \filldraw (6) circle (0.06) 
      node[below right] {$v_n$};
      
%\path[->,shorten >=1pt,shorten <=1pt] (1) edge node[fill=white] {$\infty$} (3);

      \path[->,freccia] (1) edge node[fill=white] {$\infty$}  (2) 
			               (2) edge node[fill=white] {$\infty$} (3)
										 (3) edge node[fill=white] {$\infty$}(4);
      \path[->,freccia,dashed] (4) edge  node[fill=white] {$\infty$} (6);

      \path[freccia] (1) edge[bend right] node[fill=white] {$\infty$}(3)
			               (1) edge[bend right=40] node[fill=white] {$\infty$}(4)
										 (2) edge[bend right] node[fill=white] {$\infty$} (4);
      \path[white] (1) edge[bend right=60] node[fill=white] {} coordinate  (7) (6)
			               (2) edge[bend right=50] node[fill=white] {} coordinate (8) (6)
										 (3) edge[bend right=40] node[fill=white] {}coordinate (9) (6);

      \path[shorten <=1pt] (1) edge[out=-60,in=180]node[fill=white] {$\infty$} (7)
			               (2) edge[out=-50,in=180] node[fill=white] {$\infty$}(8)
										 (3) edge[out=-40,in=180]node[fill=white] {$\infty$} (9);

      \path[->,dashed,shorten >=1pt] (7) edge[out=0,in=240] node[fill=white] {}(6)
			               (8) edge[out=0,in=230]node[fill=white] {} (6)
										 (9) edge[out=0,in=220]node[fill=white] {}  (6);
										
			\path[freccia,dashed] (4) edge[bend right,dashed] node[fill=white] {$\infty$} (6);

\end{tikzpicture}
\end{center} 
\vspace{-20pt}
\caption{The graph $M_n$ of $C(\C{\rm P}^n_q)$.}
\label{cpn}
\end{figure}
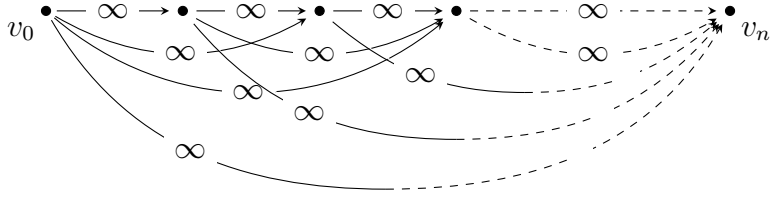
From the equivariance of the diagram~\eqref{vspull}, we conclude:
\begin{prop}\label{cppull}
The C*-algebra of the Vaksman--Soibelman quantum complex projective space 
$C(\mathbb{C}\mathrm{P}^n_q)$ has the following 
pullback structure
\begin{equation}\label{eq:4}
\begin{gathered}
\xymatrix{
&C(\mathbb{C}\mathrm{P}^n_q) \ar[ld]_{\widetilde{\pi_1^n}} \ar[rd]^{\widetilde{f^n}}\\
C(\mathbb{C}\mathrm{P}^{n-1}_q) \ar[rd]_{\iota^n}& &C(B^{2n}_q),\ar[ld]^{\pi_2^n}\\
& C(S^{2n-1}_q)}
\end{gathered}
\end{equation}
where all the maps above are analogous to the ones in the diagram \eqref{eq:pbfixed}. 
Furthermore, we obtain
\begin{equation}
 \label{eq:K0proj}
K_0(C(\mathbb{C}{\rm P}^n_q))\cong 
 K_0(C(\C{\rm P}^{n-1}_q))\oplus \partial_{10}(K_1(C(\bS^{2n-1}_q))),
\end{equation}
where $\partial_{10}$ is Milnor's connecting homomorphism.
\end{prop}
\begin{proof}
The first part of the statement follows from Corrollary \ref{main_fixed}. Recall from \cite{VS91} and \cite{HS06,HS08} that, for all $n$, 
we have that $K_0(C(S^{2n-1}_q))=\Z[1]$, $K_1(C(S^{2n-1}_q))\cong\Z$, 
and that $K_0(C(B^{2n}_q))=\mathbb{Z}[1]$ and $K_1(C(B^{2n}_q))=0$.
Note also that there is a short exact sequence (e.g., see \cite{HS02})
\begin{equation}\label{kext}
0\longrightarrow \overline{I(v_n)}\cong\mathcal{K}\longrightarrow C(\C{\rm P}^n_q)\stackrel{\widetilde{\pi_1^n}}{\longrightarrow}
C(\C{\rm P}^{n-1}_q)\longrightarrow 0,
\end{equation}
where $\mathcal{K}$ is the C*-algebra of compact operators. Now, the associated six-term exact sequence and the vanishing of 
$K_1(\mathcal{K})$
imply that $K_0(C(\C{\rm P}^n_q))\cong\Z^{n+1}$ and \mbox{$K_1(C(\C{\rm P}^n_q))=0$}, and that
the map $\widetilde{\pi_1^n}_{*}:K_0(C(\C{\rm P}^n_q))\to K_0(C(\C{\rm P}^{n-1}_q))$ is surjective.

Let us consider the Mayer--Vietoris six-term exact sequence associated to the pullback 
diagram \eqref{eq:4}: 
\begin{equation}\label{6term1}
\xymatrix{
K_0(C(\mathbb{C}\mathrm{P}^n_q)) \ar[r]^{\hspace*{-14mm}\left(\widetilde{\pi_1^n}_*\,,\widetilde{f^n}_*\right)}
& K_0(C(\mathbb{C}\mathrm{P}^{n-1}_q))\oplus K_0(C(B^{2n}_q)) \ar[r]
& K_0(C(S^{2n-1}_q)) \ar[d]\\
K_1(C(S^{2n-1}_q))\ar[u]^{\partial_{10}}
& K_1(C(\mathbb{C}\mathrm{P}^{n-1}_q))\oplus K_1(C(B^{2n}_q)) \ar[l]
& K_1(C(\mathbb{C}\mathrm{P}^n_q)). \ar[l]}
\end{equation}
We are going to prove formula \eqref{eq:K0proj} by extracting from \eqref{6term1} the following split short exact sequence:
\begin{equation}\label{newexact}
0\longrightarrow K_1(C(S^{2n-1}_q))\stackrel{\partial_{10}}{\longrightarrow} K_0(C(\C{\rm P}^n_q))
\stackrel{\widetilde{\pi_1^n}_*}{\longrightarrow} K_0(C(\C{\rm P}^{n-1}_q))\longrightarrow 0\,. 
\end{equation}
We already know that $\widetilde{\pi_1^n}_{*}$ is surjective, so to prove the exactness of \eqref{newexact},
it suffices to show that the kernel of $(\widetilde{\pi_1^n}_*\,,\widetilde{f^n}_*)$ 
is the same as the kernel of 
$\widetilde{\pi_1^n}_*$. To this end, since
$$
 \ker\left(\widetilde{\pi_1^n}_*\,,\widetilde{f^n}_*\right)=\ker\widetilde{\pi_1^n}_*\cap\ker\widetilde{f^n}_*,
$$
we want to show the inclusion $\ker\widetilde{\pi_1^n}_*\subseteq\ker\widetilde{f^n}_*$.
It follows from the pullback diagram \eqref{eq:4} and the functoriality of K-theory:
$$
\ker\widetilde{\pi_1^n}_*\subseteq \ker(\iota^n_*\circ\widetilde{\pi_1^n}_*)= \ker(\widetilde{f^n}_*\circ 
(\pi_2^n)_*)=\ker(\widetilde{f^n}_*).
$$
Here the last equality holds because $(\pi_2^n)_*$ is an isomorphism. Finally, the exact sequence \eqref{newexact} 
splits by the freeness of the $\Z$-module~$K_0(C(\C{\rm P}^{n-1}_q))$.
\end{proof}

\subsubsection{Milnor's clutching construction for generators of $K_0(C(\mathbb{C}{\rm P}^n_q))$}\label{expgen}

Recall that there are $(n+1)$-many projections $P_0$, $P_1$, ..., $P_n$\,, in the graph $M_n$ whose graph algebra
is~$C(\mathbb{C}{\rm P}^n_q)$.
Therefore, since $K_0(C(\C{\rm P}^n_q))\cong\Z^{n+1}$ and  the $K_0$-group of a graph C*-algebra is generated
by its vertex projections (see \cite[Proposition~3.8 (1)]{CET12}), we infer that
\begin{equation}\label{basis}
K_0(C(\mathbb{C}{\rm P}^{n}_q))=\mathbb{Z}[P_0]\oplus\mathbb{Z}[P_1]\oplus\ldots\oplus\mathbb{Z}[P_n].
\end{equation}
%The graph representing $C(S^{2n-1}_q)$ is a subgraph of the graph representing $C(B^{2n}_q)$
%and by a slight abuse of notation we denote the Cuntz--Krieger family of the latter by $\{S,P\}$ as well.

We will now compute the explicit value of Milnor's connecting homomorphism $\partial_{10}$
on a generator of $K_1(C(S^{2n-1}_q))$.
Let us first recall that, by Proposition~\ref{prop:K_1}, the generator of $K_1(C(S^{2n-1}_q)) \cong \Z$ is 
given by the $K_1$-class of the unitary
$$
U=S_{n-1}+(1-P_{n-1}).
$$
Next, using the pullback structure of $C(\mathbb{C}{\rm P}^n_q)$, we  follow  \cite[Section~2.1]{DHHMW12}, 
that is we find $C,D\in C(B^{2n}_q)$ 
such that $\pi_2(C)=U$ and $\pi_2(D)=U^*$:
$$
C=S_{n-1}+S_{n-1,n}+(1-P_{n-1}-P_n),
$$
$$
D=C^*=S^*_{n-1}+S^*_{n-1,n}+(1-P_{n-1}-P_n),
$$
and compute the following $2$ by $2$ matrix with entries in $C(\mathbb{C}{\rm P}^{n}_q)$:
\begin{equation}\label{idem}
p_{U}=
\begin{bmatrix}
    (1,C(2-DC)D)      & (0,C(2-DC)(1-DC)) \\
    (0,(1-DC)D)       & (0,(1-DC)^2)
\end{bmatrix}
=
\begin{bmatrix}
    (1,1-P_n)      & (0,0) \\
    (0,0)       & (0,0)
\end{bmatrix}.
\end{equation}
By \cite[Theorem~2.2]{DHHMW12}, the element $\partial_{10}([U])=[p_{U}]-[1]$ 
is a generator of $K_0(C(\mathbb{C}{\rm P}^{n}_q))$.
Observe that in the above formula for $p_{U}$ we can remove all the entries except the top left one without changing its
class in $K_0(C(\mathbb{C}{\rm P}^{n-1}_q))$, namely
$[1]-[p_{U}]=[1]-[p]$,
where $p:=(1,1-P_n)=(1,1)-(0,P_n)$
is a projection in $C(\mathbb{C}{\rm P}^{n}_q)$. Furthermore, since $p$ and $1-p$ are orthogonal, we
have $[p]+[1-p]=[1]$, so
\begin{equation}\label{milnor}
-\partial_{10}([U])=[1]-[p]=[1-p]=[(0,P_n)]=[P_n].
\end{equation}
Here the rightmost projection $P_n$ is viewed as the vertex projection corresponding to the vertex $v_n$ in the graph of 
$C(\mathbb{C}{\rm P}^n_q)$, and the rightmost equality follows from the fact that the isomorphism
from $C(\mathbb{C}{\rm P}^n_q)$ to the pullback C*-algebra of the digaram \eqref{eq:4} maps $P_n$ to $(0,P_n)$. 

Finally, let us observe that the Milnor's idempotent in the above calculation is 
$$
p_U\cong 1-P_n=P_0+P_1+\ldots+P_{n-1}\,.
$$
Hence the projective module it defines can be understood as the section module of
a noncommutative vector bundle obtained by the Milnor clutching construction.

\subsection{Quantum lens spaces  and quantum teardrops}

Quantum lens spaces, both weighted and unweighted, have been the subject of increasing interest in the last years. Their 
realisation as graph C*-algebras has been first proven in \cite{HS03}, and then further generalised in \cite{BS16} under less stringent 
assumptions. In the rest of the paper, we focus on the three-dimensional quantum lens spaces  
$L^3_q(l;1,l)$. 
%It was shown in \cite[Theorem 3.3]{BF12} that their coordinate algebras $\mathcal{O}(L_q^3(l; 1,l))$ admit a structure of
%principal $\mathcal{O}(U(1))$-comodule algebras over the coordinate algebras $\mathcal{O}(WP^1_q(1,l))$ of quantum teardrops. 
%This result was later extended to their C*-algebraic completions in \cite{AKL16} using the language of Cuntz--Pimsner algebras, 
%resulting in an algorithmic computation of the K-theory and K-homology groups of large class of three dimensional lens spaces. 
In~\cite{DL14},
 generators for the K-theory and K-homology of multi-dimensional quantum weighted projective spaces were constructed, leading 
to an extension of the K-theoretic computations for quantum 
weighted lens spaces in \cite{A18} in the Cuntz--Pimsner picture \cite{ADL16}.

\subsubsection{Lens spaces.}

Our starting point is the  C*-algebra $C(L^3_q(l;1,l))$ of the quantum lens space $L^3_q(l;1,l)$.
As explained in 
\cite[Example 2.1]{BS16}, $C(L^3_q(l;1,l))$ can be viewed as the graph C*-algebra of the
graph $L^3_l$ (see Figure \ref{fig:lens_3l}) with
\vspace*{-2mm}
\begin{list}{{\tiny\raisebox{1pt}{$\blacksquare\;$}}}{\leftmargin=2em \itemsep=2pt}
\item $l+1$ vertices $\{v^0_0,v^1_0,\ldots,v^1_{l-1}\}$,
\item one edge $e^{01}_{i}$ from $v^0_0$ to $v^1_i$ for all $0\leq i\leq l-1$,
\item a loop $e^0_0$ over the vertex $v^0_0$ and one loop $e^1_i$ over each vertex $v^1_i$ 
for all $0\leq i\leq l-1$.
\end{list}
Observe that $C^*(L^3_1)\cong C(S^3_q)$ (e.g., see~\cite{HS02}).
Throughout this subsection, we denote the Cuntz--Krieger $L^3_l$-family by $\{S,P\}$. 
To simplify notation we set 
$S_{e^{01}_j}:=S^{01}_j$, $S_{e^i_j}:=S^i_j$ and $P_{v^i_j}:=P^i_j$, 
where $0\leq i \leq1$ and $0\leq j\leq l-1$.

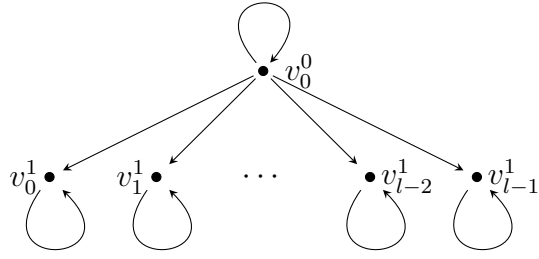
\begin{figure}[h]
\begin{center}
\begin{tikzpicture}[>=stealth,node distance=40pt,main node/.style={circle,inner sep=2pt},
freccia/.style={->,shorten >=1pt,shorten <=1pt},
ciclo/.style={out=130, in=50, loop, distance=40pt, ->},
ciclo2/.style={out=235, in=315, loop, distance=40pt, ->}]

      \node[main node] (1) {};
      \node (2) [below of=1] {$\cdots$};
      \node (3) [left of=2]{};
      \node (4) [left of=3]{};
      \node (5) [right of=2]{};
     \node (6) [right of=5] {};

      \filldraw (1) circle (0.06) node[right] {$\ v_0^0$};
      \filldraw (4) circle (0.06) node[left] {$v_0^1$};
      \filldraw (3) circle (0.06) node[left] {$v_1^1$};
      \filldraw (5) circle (0.06) node[right] {$v_{l-2}^1$};
      \filldraw (6) circle (0.06) node[right] {$v_{l-1}^1$};

      \path[freccia] (1) edge[ciclo] (1);
			\path[freccia] (4) edge[ciclo2] (4);
			\path[freccia] (3) edge[ciclo2] (3);
			\path[freccia] (5) edge[ciclo2] (5);
			\path[freccia] (6) edge[ciclo2] (6);

      \path[freccia] (1) edge (4);         							
\path[freccia] (1) edge (3)	; 
\path[freccia] (1) edge (5)	;           
\path[freccia] (1) edge (6)	;

\end{tikzpicture}
\end{center}
\vspace{-20pt}
\caption{The graph  $L^3_l$.}
\label{fig:lens_3l}
\end{figure} 
Clearly, the graph $L^3_l$ is $v^1_j$-trimmable for $0\leq j\leq l-1$. Let us choose the vertex $v^1_{l-1}$, and construct a new 
graph by removing the loop $e^1_{l-1}$. We denote the thus obtained graph by $Q_l$ (see Figure~\ref{fig:Ql}). 
Note that the graph $Q_l$ and its C*-algebra do not depend on our 
choice of a vertex.

\begin{figure}[h]
\begin{center}
\begin{tikzpicture}[>=stealth,node distance=40pt,main node/.style={circle,inner sep=2pt},
freccia/.style={->,shorten >=1pt,shorten <=1pt},
ciclo/.style={out=130, in=50, loop, distance=40pt, ->},
ciclo2/.style={out=235, in=315, loop, distance=40pt, ->}]

      \node[main node] (1) {};
      \node (2) [below of=1] {$\cdots$};
      \node (3) [left of=2]{};
      \node (4) [left of=3]{};
      \node (5) [right of=2]{};
     \node (6) [right of=5] {};

      \filldraw (1) circle (0.06) node[right] {$\ v_0^0$};
      \filldraw (4) circle (0.06) node[left] {$v_0^1$};
      \filldraw (3) circle (0.06) node[left] {$v_1^1$};
      \filldraw (5) circle (0.06) node[right] {$v_{l-2}^1$};
      \filldraw (6) circle (0.06) node[right] {$v_{l-1}^1$};

			\path[freccia] (1) edge[ciclo] (1);
			\path[freccia] (4) edge[ciclo2] (4);
			\path[freccia] (3) edge[ciclo2] (3);
			\path[freccia] (5) edge[ciclo2] (5);
			%\path[freccia] (6) edge[ciclo2] (6);

      \path[freccia] (1) edge (4);         							
\path[freccia] (1) edge (3)	; 
\path[freccia] (1) edge (5)	;           
\path[freccia] (1) edge (6)	;

\end{tikzpicture}
\end{center}
\vspace{-20pt}
\caption{The graph $Q_l$.}
\label{fig:Ql}
\end{figure}
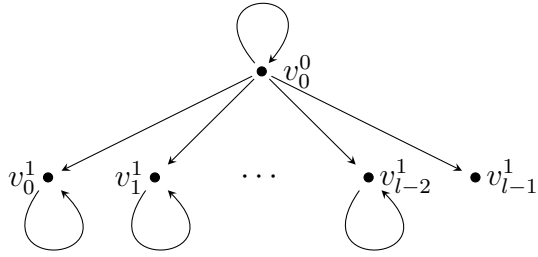

%The graph $Q_l$ is a one-sink extension of the graph $L^3_{l-1}$ of the quantum lens space $L^3_q(l-1;1,l-1)$. We can compute its K-theory using \cite[Lemma 5.2]{RTW04} and obtain
%\[K_0(C^*(Q_l))= \coker((A^t-1)\oplus W)= \Z^{l+1}, \qquad K_1(C^*(Q_l))= \ker ((A^t-1)\oplus W) = \Z^{l}.\] 

By Theorem \ref{main}, we obtain the following $U(1)$-equivariant pullback structure of the algebra 
$C(L^3_q (l;1,l))\cong C^*(L^3_l)$:
\begin{equation}\label{lenspull}
\begin{gathered}
\xymatrix{
& C^*(L^3_l) \ar[ld] \ar[rd] \\
 C^*(L^3_{l-1}) \ar[rd]
& 
& C^*(Q_l)\otimes C(S^1). \ar[ld]\\
& C^*(L^3_{l-1})\otimes C(S^1)
}
\end{gathered}
\end{equation}
Here all the maps are analogous to the ones used in Theorem~\ref{main}.

\subsubsection{Teardrops.}\label{teardrops}

Recall that the C*-algebra  $C(\mathbb{W}{\rm P}^1_q(1,l))$ of the weighted projective space $\mathbb{W}{\rm P}^1_q(1,l)$ 
\cite{BF12} is defined as the $U(1)$-fixed-point subalgebra of $C(L^3_q(l;1,l))$. 
It~can be viewed as the graph C*-algebra of the graph $W_n$
 consisting of the same vertices \mbox{$v^0_0$, ..., $v^1_{l-1}$} as in the graph $L^{3}_l$ (vertex projections are gauge 
 invariant), with no loops, and countably many edges between $v^0_0$ and all the other vertices~\cite{BS16}. 
 (See Figure~\ref{fig:wl}.)%, where thick arrows denote infinitely many edges.)

\begin{figure}[h]
\begin{center}
\begin{tikzpicture}[>=stealth,node distance=40pt,main node/.style={circle,inner sep=2pt},
freccia/.style={->,shorten >=1pt,shorten <=1pt},
ciclo/.style={out=130, in=50, loop, distance=40pt, ->},
ciclo2/.style={out=235, in=315, loop, distance=40pt, ->}]

      \node[main node] (1) {};
      \node (2) [below of=1] {$\cdots$};
      \node (3) [left of=2]{};
      \node (4) [left of=3]{};
      \node (5) [right of=2]{};
     \node (6) [right of=5] {};

      \filldraw (1) circle (0.06) node[right] {$\ v_0^0$};
      \filldraw (4) circle (0.06) node[left] {$v_0^1$};
      \filldraw (3) circle (0.06) node[left] {$v_1^1$};
      \filldraw (5) circle (0.06) node[right] {$v_{l-2}^1$};
      \filldraw (6) circle (0.06) node[right] {$v_{l-1}^1$};

			%\path[freccia] (1) edge[ciclo] (1);
			%\path[freccia] (4) edge[ciclo2] (4);
			%\path[freccia] (3) edge[ciclo2] (3);
			%\path[freccia] (5) edge[ciclo2] (5);
			%\path[freccia] (6) edge[ciclo2] (6);

\path[->,shorten >=1pt,shorten <=1pt] (1) edge node[fill=white] {$\infty$} (3);

\path[->,shorten >=1pt,shorten <=1pt] (1) edge node[fill=white] {$\infty$} (4);

\path[->,shorten >=1pt,shorten <=1pt] (1) edge node[fill=white] {$\infty$} (5);

\path[->,shorten >=1pt,shorten <=1pt] (1) edge node[fill=white] {$\infty$} (6);

\end{tikzpicture}
\end{center}
\vspace{-20pt}
\caption{The graph $W_l$ of $C(\mathbb{W}{\rm P}^1_q(1,l))$.}
\label{fig:wl}
\end{figure}
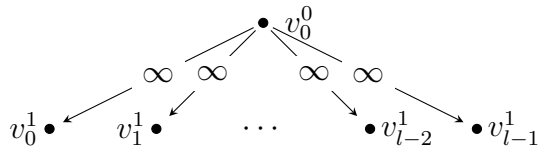

As discussed in Section \ref{sec:equiv}, from \eqref{lenspull} we obtain another pullback diagram of fixed-point 
subalgebras:
\begin{equation}\label{wppull}
\begin{gathered}
\xymatrix{
& C(\mathbb{W}{\rm P}^1_q(1,l)) \ar[ld] \ar[rd] \\
 C(\mathbb{W}{\rm P}^1_q(1,l-1)) \ar[rd]
& 
& C^*(Q_l). \ar[ld]_{\chi_1}\\
& C^*(L^3_{l-1})
}
\end{gathered}
\end{equation}
Here $\chi_1$ is the quotient map by the ideal $\overline{I(v^1_{l-1})}$. 

Let us introduce two more maps needed for Lemma~\ref{qlpb} below.
Define $\chi_2:C^*(Q_l)\to C^*(Q_l/H)$ to be the quotient map given by the ideal generated by the saturated hereditary 
subset $H=\{v^1_0, v^1_1, \ldots, v_{l-2}^1\}$, and $g:C^*(L^3_{l-1})\to C^*(L^3_l/H)$ the quotient map given by the ideal 
generated by $H$ 
considered as a subset in $(L^3_{l-1})_0$.
\begin{lemma}\label{qlpb}
The following diagram of the above-defined $U(1)$-equivariant *-ho\-mo\-mor\-phisms
\begin{equation}\label{qlpull}
\begin{gathered}
\xymatrix{
& C^*(Q_l) \ar[ld]_{\chi_1} \ar[rd]^{\chi_2} \\
C^*(L^3_{l-1}) \ar[rd]_g
& 
& \mathcal{T} \ar[ld]^{\sigma}\\
& C(S^1)
}
\end{gathered}
\end{equation}
is a pullback diagram. Here $\mathcal{T}$ is the Toeplitz algebra and $\sigma:\mathcal{T}\to C(S^1)$ is the symbol 
map. 
 \end{lemma}
\begin{proof}
Recall that the Toeplitz algebra $\mathcal{T}$ is isomorphic to the graph C*-algebra 
corresponding to the quotient graph $Q_l/H$ and the isomorphism is given by 
$s\mapsto S^0_0+S^{01}_{l-1}$ (e.g., see \cite{HS08}), where $s$ is the isometry generating $\mathcal{T}$.

By \cite[Proposition~3.1]{P99} and the surjectivity of $g$ and $\sigma$, it suffices to show that
$\ker\chi_1\cap\ker\chi_2=\{0\}$ and that $\chi_2(\ker\chi_1)\subseteq\ker\sigma$. To prove the first condition,
recall that, since $\ker\chi_1$ and $\ker\chi_2$ are closed ideals in a~C*-algebra, we have that
$\ker\chi_1\cap\ker\chi_2=\ker\chi_1\ker\chi_2$.
Next, $\{v^1_{l-1}\}$ and $H$ are saturated hereditary subsets of $(Q_l)_0$, so
$$
\ker\chi_1=\overline{I_{Q_l}(v_{l-1}^1)}
\qquad\text{and}\qquad
\ker\chi_2=\overline{I_{Q_l}(H)}.
$$
Using (\ref{herideal}), one can observe that an arbitrary element of $\ker\chi_1\ker\chi_2$ is of the form 
$S_\alpha S_\beta^*S_\gamma S_\delta^*$, where $\alpha, \beta\in \text{Path}(Q_l)$ with $r(\alpha)=r(\beta)=v_{l-1}^1$
and $\gamma,\delta\in \text{Path}(Q_l)$ with $r(\gamma)=r(\delta)\in H$. The claim follows from the analysis of all
possible paths satisfying the above conditions.

To prove the second condition, notice that $\ker\sigma=\overline{I_{Q_l/H}(v_{l-1}^1)}$.
Any element of $I_{Q_l/H}(v_{l-1}^1)$ is an element of $I_{Q_l}(v_{l-1}^1)$,
and $\chi_2(S_\alpha)=S_\alpha$ for all $\alpha\in{\rm Path}(Q_l/H)$. Hence
$\chi_2(I_{Q_l}(v_{l-1}^1))\subseteq I_{Q_l/H}(v^1_{l-1})$.
Furthermore, since $\chi_2$ is a *-homomorphism, we can argue as in the proof of Theorem \ref{main} to conclude that 
$\chi_2(\ker\chi_1)\subseteq\ker\sigma$. 
\end{proof}
\begin{rem}
{\em 
Note that for $l=2$, we get the graph $Q_2$ considered in Section~\ref{toe}. By~Lemma~\ref{qlpb},
the C*-algebra $C^*(Q_2)$ has the following $U(1)$-equivariant pullback structure:
\begin{equation}
\begin{gathered}
\xymatrix{
& C^*(Q_l) \ar[ld] \ar[rd] \\
C(S^3_q) \ar[rd]
& 
& \mathcal{T} .\ar[ld]\\
& C(S^1)
}
\end{gathered}
\end{equation}
The  equivariance of the above diagram allows it to descend to the fixed-point subalgebras:
\begin{equation}
\begin{gathered}
\xymatrix{
& C^*(Q_l)^{U(1)} \ar[ld]\ar[rd] \\
C(\C{\rm P}^1_q) \ar[rd]
& 
& \mathcal{T}^{U(1)}. \ar[ld]\\
& \C
}
\end{gathered}
\end{equation}
The Mayer--Vietoris six-term exact sequence in K-theory associated to this diagram gives the K-groups of $C^*(Q_2)^{U(1)}$
as in Proposition~\ref{toepro}.}\hfill$\lozenge$
\end{rem}
Next, using Lemma~\ref{qlpb} along with \cite[Proposition~2.7]{P99}
and the analogous reasoning as in the proof of Lemma~\ref{cppull}, 
we arrive at:
\begin{prop}
The C*-algebra  $C(\mathbb{W}\mathrm{P}^1_q(1,l))$ 
has the following pullback structure
\begin{equation}\label{wp2pull}
\begin{gathered}
\xymatrix{
& C(\mathbb{W}{\rm P}^1_q(1,l)) \ar[ld] \ar[rd] \\
 C(\mathbb{W}{\rm P}^1_q(1,l-1)) \ar[rd]
& 
& \mathcal{T}. \ar[ld]\\
& C(S^1)
}
\end{gathered}
\end{equation}
Furthermore, we obtain
\begin{equation}\label{K0wproj}
K_0(C(\mathbb{W}{\rm P}^1_q(1,l))) \cong K_0(C(\mathbb{W}{\rm P}^1_q(1,l-1)))\oplus \partial_{10}
(K_1(C(\bS^{1}))),
\end{equation}
where $\partial_{10}$ is Milnor's connecting homomorphism.
\end{prop}

\subsubsection{Milnor's clutching construction for generators of $K_0(C(\mathbb{W}{\rm P}^1_q(1,l)))$}
Recall that there are $(l+1)$-many projections $P^0_0$, $P^1_0$, ..., $P^1_{l-1}$\,, 
in the graph $W_n$ whose graph algebra is~$C(\mathbb{W}{\rm P}^1(1,l))$.
Therefore, since $K_0(C(\mathbb{W}{\rm P}^1(1,l)))\cong\Z^{l+1}$ and  the $K_0$-group of a graph C*-algebra is generated
by its vertex projections (see \cite[Proposition~3.8 (1)]{CET12}), we infer that
\begin{equation}\label{basis:wp}
K_0(C(\mathbb{W}{\rm P}^1(1,l)))=\mathbb{Z}[P^0_0]\oplus\mathbb{Z}[P^1_0]\oplus\ldots\oplus\mathbb{Z}[P^1_{l-1}].
\end{equation}

Below we compute the value of Milnor's connecting homomorphism on the generator $[u]\in K_1(C(S^1))$, where
$u$ is the standard generator of~$C(S^1)$. This computation
closely follows an analogous computation in Section~\ref{expgen}.
First, we find $c,d\in \mathcal{T}$ 
such that $\sigma(c)=u$ and $\sigma(d)=u^*$:
$$c=S^0_0+S^1_{l-1},\quad d=c^*=\left(S^0_0\right)^*+(S^1_{l-1})^*,$$ 
and, using the formula \eqref{idem}, we compute the following $2$ by $2$ matrix with entries in $C(\mathbb{W}{\rm P}^1_q(1,l))$:
$$
p_u=
\begin{bmatrix}
    (1,P_0^0)      & (0,0) \\
    (0,0)       & (0,0)
\end{bmatrix}.
$$
The element $\partial_{10}([u])=[p_u]-[1]$ is a generator of $K_0(C(\mathbb{W}{\rm P}^1_q(1,l)))$. We have that
$[1]-[p_u]=[1]-[p]$, where $p:=(1,P^0_0)=(1,1)-(0,P^1_{l-1})$
is a projection in $C(\mathbb{W}{\rm P}^{1}_q(1,l))$. Notice that
$$
-\partial_{10}([u])=[1]-[p]=[(0,P_{l-1}^1)]=[P_{l-1}^1],
$$
where $P_{l-1}^1$ is viewed as an element of $C(\mathbb{W}{\rm P}^{1}_q(1,l))$.

Finally, as we did in Section~\ref{expgen}, let us observe that the Milnor's idempotent in the above calculation is 
$$
p_u\cong 1-P^1_{l-1}=P^0_0+P^1_0+\ldots+P^1_{l-2}\,.
$$

\subsection*{Acknowledgements}\noindent
This work is part of the project \emph{Quantum Dynamics} partially supported by
the EU-grant H2020-MSCA-RISE-2015-691246 and by the Polish Government grants
3542/H2020/2016/2 and \mbox{328941/PnH/2016}. It was initiated while F.D., P.M.H., 
and M.T. were visiting Penn State 
University, and they  are very grateful to the University for  excellent 
working conditions and financial support.
P.M.H. and M.T. are also very thankful to the University of Copenhagen for its financial support and amazing
hospitality.
%, and to Wojciech Szyma\'nski for enlightening discussions. 
Furthermore, F.A., P.M.H., and M.T. are happy to
acknowledge a substantial logistic support of the Max 
Planck Insitute for Mathematics in the Science in Leipzig, where much of joint research was carried out. 
Most importantly,
our deepest gratitude goes to the graph C*-algebra gurus S\o ren Eilers and Wojciech Szyma\'nski 
for their key technical support.
Finally, F.A. would also like to thank Magnus Goffeng and Bram Mesland for valuable conversations 
about Cuntz--Krieger algebras.

\end{document}